\documentclass[12pt]{article}
\usepackage{amsmath,amssymb,amsthm}
\date{}

\newcommand{\R}{{\mathbb R}}
\newcommand{\N}{{\mathbb N}}
\newcommand{\D}{{\mathcal D}}
\newcommand{\supp}{\mathop{\rm supp}}
\renewcommand{\div}{\mathop{\rm div}\nolimits}
\newcommand{\codim}{\mathop{\rm codim}}
\renewcommand{\Im}{\mathop{\rm Im}}
\numberwithin{equation}{section}
\theoremstyle{plain} % Далее вводятся окружения типа "теорема"
\newtheorem{theorem}{Theorem}[section]
\newtheorem{lemma}{Lemma}[section]

\newtheorem{corollary}{Corollary}[section]
\theoremstyle{definition}
\newtheorem{definition}{Definition}[section]
\newtheorem{remark}{Remark}[section]

\begin{document}
\title{On evolutionary equations related to skew-symmetric spatial operators}
\author{Evgeny Yu. Panov \\ St. Petersburg Department of V.\,A.~Steklov Institute \\ of Mathematics of the Russian Academy of Sciences, \\
St. Petersburg, Russia}
\maketitle

\begin{abstract}
We study generalized solutions of an evolutionary equation related to some densely defined skew-symmetric operator in a real Hilbert space. We establish existence of a contractive semigroup, which provides generalized solutions, and suggest a criteria of uniqueness of this semigroup. We also find a stronger criteria of uniqueness of generalized solutions.   Applications to transport equations with  solenoidal (and generally discontinuous) coefficients are given.
\end{abstract}

\section{Introduction}

Let $H$ be a real Hilbert space, $A_0$ be a skew-symmetric linear operator in $H$ defined on a dense domain
$X_0=D(A_0)\subset H$. We recall that the operator $A_0$ is skew-symmetric if $(A_0u,v)=-(u,A_0v)$ $\forall u,v\in X_0$, where $(\cdot,\cdot)$ denotes the inner product in $H$.
Alternatively, the skew-symmetricity may be described as $-A_0\subset (A_0)^*$, where $(A_0)^*$ is the adjoint operator to $A_0$. Since an adjoint operator is closed, there exists the closure $A$ of the operator $A_0$. This operator $A$ is a closed skew-symmetric operator, and $A^*=(A_0)^*$.
We underline that in view of the identity $2((A_0u,v)+(u,A_0v))=(A_0(u+v),u+v)-(A_0(u-v),u-v)$ the skew-symmetricity of $A_0$ is equivalent to the condition $(A_0u,u)=0$ $\forall u\in X_0$.

The operator $A^*$ is an extension of the operator $-A$, and may be not skew-symmetric. Skew-symmericity of $A^*$ holds only in the case $A^*=-A$, that is, when the operator $A$ is skew-adjoint.

We consider the evolutionary equation
\begin{equation}\label{e1}
u'-A^*u =0, \quad u=u(t)\in L^\infty_{loc}(\R_+,H),
\end{equation}
where $\R_+=[0,+\infty)$, with the initial condition
\begin{equation}\label{c1}
u(0)=u_0\in H.
\end{equation}

\begin{definition}\label{def1}
An $H$-valued function $u=u(t)\in L^\infty_{loc}(\R_+,H)$ is called a generalized solution (g.s.) of problem (\ref{e1}), (\ref{c1}) if
$\forall f(t)\in C_0^1(\R_+,X_0)$, where the space $X_0$ equipped with the graph norm $\|u\|^2=\|u\|_H^2+\|A_0u\|_H^2$,
\begin{equation}\label{gs}
\int_{\R_\pm}(u(t),f'(t)+A_0f(t))dt+(u_0,f(0))=0.
\end{equation}
\end{definition}

\begin{remark}\label{rem1}

(1) It follows from relation (\ref{gs}) with $f(t)=v\varphi(t)$, $\varphi(t)\in C_0^1((0,+\infty))$, $v\in X_0$, that
$\frac{d}{dt} (u(t),v)=(u(t),Av)$ in $\D'((0,+\infty))$, where $\D'(I)$ denotes the space of distributions on an open set $I$. This relation readily implies weak continuity of $u(t)$ (after possible correction on a set of null measure).
Besides, (\ref{gs}) implies initial condition (\ref{c1}), understood in the sense of weak convergence $u(t)\rightharpoonup u_0$ as $t\to 0+$;

(2) Let $X=D(A)$ be the domain of the operator $A$, equipped with the graph norm $\|u\|^2=\|u\|_H^2+\|A(u)\|_H^2$, $X$ is a Banach space in view of the closedness of $A$. Since the operator $A$ is the closure of $A_0$, the space $X_0$ is dense both in $H$ and in $X$. In turn, this implies density of the space
$C_0^1(\R_+,X_0)$ in $C_0^1(\R_+,H)\cap C_0(\R_+,X)$. Therefore, every test function $f(t)\in C_0^1(\R_+,H)\cap C_0(\R_+,X)$ can be approximated by a sequence
$f_r(t)\in C_0^1(\R_+,X_0)$, $r\in\N$, so that $f_r(t)\to f(t)$, $f_r'(t)\to f'(t)$, $A_0f_r(t)\to Af(t)$ as $r\to\infty$ in $H$, uniformly with respect to $t$. Passing to the limit as $r\to\infty$ in relations (\ref{gs}) with the test functions $f=f_r$, we arrive at the relation
\[
\int_{\R_\pm}(u(t),f'(t)+Af(t))dt+(u_0,f(0))=0.
\]
Thus, (\ref{gs}) remains true for all $f(t)\in C_0^1(\R_+,H)\cap C_0(\R_+,X)$. Since $C_0^1(\R_+,H)\cap C_0(\R_+,X)\supset C_0^1(\R_+,X)$, then we may replace $A_0,X_0$ by $A,X$ in Definition~\ref{def1} and could initially assume that $A_0=A$ is a closed skew-symmetric operator.
\end{remark}

\section{One example}

Let $a(x)=(a_1(x),\ldots,a_n(x)) \in L^2_{loc}(\R^n,\R^n)$ be a solenoidal vector field on $\R^n$, that is,
\begin{equation}\label{sol}
\div a(x)=0
\end{equation}
in the sense of distributions on $\R^n$ (in $\D'(\R^n)$). We consider the transport equation
\[
u_t+a(x)\cdot\nabla_x u=u_t+\sum_{i=1}^n a_i(x)u_{x_i}=0
\]
(here $p\cdot q$ denotes the usual scalar multiplication of vectors $p,q\in\R^n$), $u=u(t,x)$, $x\in\R^n$, $t>0$.
By the solenoidality condition this equation can be written (at least formally) in the conservative form
\begin{equation}\label{tre}
u_t+\div_x (au)=0.
\end{equation}
This allows to introduce the notion of g.s. $u(t,x)$ to the Cauchy problem for equation
(\ref{tre}) with initial data
\begin{equation}\label{tri}
u(0,x)=u_0(x)\in L^2(\R^n).
\end{equation}

\begin{definition}\label{def2}
A function $u=u(t,x)\in L^\infty_{loc}(\R_+,L^2(\R^n))$ is called a g.s. of problem (\ref{tre}), (\ref{tri}) if for each
$f=f(t,x)\in C_0^1(\R_+\times\R^n)$
\begin{equation}\label{tr3}
\int_{\Pi} u[f_t+a\cdot\nabla_x f]dtdx+\int_{\R^n} u_0(x)f(0,x)dx=0.
\end{equation}
\end{definition}

We introduce the real Hilbert space $H=L^2(\R^n)$ and the unbounded linear operator $A_0$ in $H$ with the domain $D(A_0)=C_0^1(\R^n)\subset H$, defined by the equality $A_0u=a(x)\cdot\nabla u(x)\in H$. This operator is skew-symmetric.
In fact, for $u,v\in C_0^1(\R^n)$
\begin{align*}(A_0u,v)_H+(u,A_0v)_H=\int_{\R^n} a(x)\cdot(v(x)\nabla u(x)+u(x)\nabla v(x))dx= \\ \int_{\R^n} a(x)\cdot\nabla (u(x)v(x))dx=0
\end{align*}
by the solenoidality requirement (\ref{sol}). As is easy to see, the adjoint operator $(A_0)^*$ is determined by the relation $(A_0)^*u=-\div (a(x)u)$ in $\D'(\R^n)$. The domain of this operator consists of such $u=u(x)\in H$ that the distribution $-\div (au)\in H$. Analyzing Definitions~\ref{def1},~\ref{def2}, we conclude that the notion of g.s. of the transport problem (\ref{tre}), (\ref{tri}) is compatible with the theory of g.s. to the abstract problem  (\ref{e1}), (\ref{c1}).

\section{Contractive semigroups of generalized solutions}
We consider a $C_0$-semigroup $T_t=e^{tB}$ of bounded linear operators in $H$ with an infinitesimal generator $B$. The operator $B$ is known to be a densely defined closed operator, and its domain $D(B)$ is a Banach space with respect to the graph norm. The following statement holds.

\begin{theorem}\label{th1}
The functions $u(t)=e^{tB}u_0$ are g.s. to the problem (\ref{e1}), (\ref{c1}) for each $u_0\in H$ if and only if
$B\subset A^*$.
\end{theorem}

\begin{proof}
First assume that $B\subset A^*$ and $u_0\in D(B)$. Then $u(t)=e^{tB}u_0\in C(\R_+,D(B))\cap C^1(\R_+,H)$ and $u'(t)=T_tBu_0=BT_tu_0=Bu(t)$. Therefore, for all $f(t)\in C_0^1(\R_+,X_0)$ there exists the continuous derivative
\begin{equation}\label{sk1}
\frac{d}{dt}(u(t),f(t))=(u'(t),f(t))+(u(t),f'(t))=(Bu(t),f(t))+(u(t),f'(t)).
\end{equation}
By our assumption $B\subset A^*$ and therefore.
\[
(Bu(t),f(t))=(A^*u(t),f(t))=(u(t),Af(t))=(u(t),A_0f(t)).
\]
In view of (\ref{sk1}), we arrive at the equality
\[
\frac{d}{dt}(u(t),f(t))=(u(t),A_0f(t)+f'(t)),
\]
which implies, after the integration, the relation
\[
\int_{\R_+}(u(t),f'(t)+A_0f(t))dt=\int_{\R_+}\frac{d}{dt}(u(t),f(t))dt=-(u_0,f(0)).
\]
Hence, the identity (\ref{gs}) holds and $u(t)$ is a g.s. of (\ref{e1}), (\ref{c1}).
In the case of arbitrary $u_0\in H$ we choose a sequence $u_{0k}\in D(B)$, convergent as $k\to\infty$ to $u_0$ in $H$. As we have already established, the function $u_k(t)=T_tu_{0k}$ is a g.s. of the problem (\ref{e1}), (\ref{c1}) with initial data $u_{0k}$ for each $k\in\N$. It is known that for a $C_0$-semigroup $T_t$ the operator norms $\|T_t\|$ are exponentially bounded, $\|T_t\|\le Ce^{\alpha t}$ for some constants $C>0$, $\alpha\in\R$. Therefore,
\[
\|u_k(t)-u(t)\|_H\le\|T_t\|\|u_{0k}-u_0\|_H\mathop{\to}_{k\to\infty} 0
\]
uniformly on any segment $[0,T]$. This allows to pass to the limit as $k\to\infty$ in the relations (\ref{gs})
\[
\int_{\R_+}(u_k(t),f'(t)+A_0f(t))dt+(u_{0k},f(0))=0,
\]
corresponding to g.s. $u_k$ and derive that for each test function $f(t)\in C_0^1(\R_+,X_0)$
\[
\int_{\R_+}(u(t),f'(t)+A_0f(t))dt+(u_0,f(0))=0.
\]
Hence, $u(t)$ is a g.s. of the problem (\ref{e1}), (\ref{c1}), as was to be proved.

Conversely, assume that the functions $u(t)=T_tu_0$ are g.s. of (\ref{e1}), (\ref{c1}) for all $u_0\in H$. If $u_0\in D(B)$, then $u(t)=T_tu_0\in C^1(\R_+,H)$ and $u'(0)=Bu_0$. Therefore, for all $v\in D(A_0)$
the scalar function  $I(t)=(u(t),v)$ lies in $C^1(\R_+)$ and $I'(0+)=(Bu_0,v)$.
On the other hand, for all $h(t)\in C_0^1(\R_+)$
\[
\int_{\R_+} I(t)h'(t)dt=\int_{\R_+} (u(t),vh'(t))dt=
-h(0)(u_0,v)-\int_{\R_+} (u(t),A_0v)h(t)dt
\]
by the equality (\ref{gs}) with $f=h(t)v$. Integration by parts then implies that
\[
\int_{\R_+} I'(t)h(t)dt=\int_{\R_+} (u(t),A_0v)h(t)dt
\]
and by arbitrariness of $h(t)\in C_0^1(\R_+)$ we find $I'(t)=(u(t),A_0v)$. In particular,
$(u_0,A_0v)=I'(0+)=(Bu_0,v)$. Since this identity holds for all $v\in D(A_0)$, then by definition of the adjoint operator $A^*=(A_0)^*$ we find that $u_0\in D(A^*)$ and $A^*u_0=Bu_0$. Hence,
$B\subset A^*$. The proof is complete.
\end{proof}

By the Lumer-Phillips theorem (see, for instance, \cite[Theorem~1.1.3]{Phil}) a $C_0$-semigroup $T_t=e^{tB}$ is contractive (i.e., $\|T_t\|\le 1$ $\forall t>0$) if and only if its generator $B$ is an $m$-dissipative operator. Remind that an operator $A$ in a Banach space $H$ is called dissipative if $\|u-hAu\|\ge \|u\|$ $\forall u\in H$, $h>0$. A dissipative operator $A$ is called $m$-dissipative if $\Im (E-hA)=X$ for all $h>0$ (it is sufficient that this property is satisfied only for one value $h>0$). It is known that an $m$-dissipative operator is a maximal dissipative operator, in the case of a Hilbert space $H$ the inverse statement is also true. Notice also that in the case of Hilbert space $H$ dissipativity of an operator $A$ reduces to the condition  $(Au,u)\le 0$ $\forall u\in H$. In particular, skew-symmetric operators are dissipative. The following property means that the requirement $B\subset A^*$ of Theorem~\ref{th1} for an
$m$-dissipative operator is equivalent to the condition that $B$ is an extension of the operator $-A$.

\begin{lemma}\label{lem1}
Let $B$ be an $m$-dissipative operator in $H$. Then $B\subset A^*\Leftrightarrow -A\subset B$.
\end{lemma}

\begin{proof}
Assume that $B\subset A^*$. Since we have also that $-A\subset A^*$, then $-Au=Bu=A^*u$ on $D(A)\cap D(B)$. This allows to construct a common extension $\tilde B$ of both operators $-A$ and $B$, which is defined on the space $D(A)+D(B)$ by the equality
\[
\tilde Bu=-Au_1+Bu_2, \quad u=u_1+u_2, \ u_1\in D(A), u_2\in D(B).
\]
Since the operators $-A$ and $B$ coincide on their common domain $D(A)\cap D(B)$, this definition is correct and does not depend on the representation $u=u_1+u_2$. By the construction we have $-A\subset \tilde B$, $B\subset \tilde B$. Let us show that the operator $\tilde B$ is dissipative. For arbitrary $u\in D(A)+D(B)$ we find such $u_1\in D(A)$, $u_2\in D(B)$ that $u=u_1+u_2$. Then
\begin{align}\label{sk2}
(\tilde Bu,u)=(-Au_1+Bu_2,u_1+u_2)=\nonumber\\-(Au_1,u_1)+(Bu_2,u_2)+((u_1,Bu_2)-(Au_1,u_2))=\nonumber\\
(Bu_2,u_2)+((u_1,A^*u_2)-(Au_1,u_2))=(Bu_2,u_2)\le 0
\end{align}
by dissipativity of $B$. We also take into account that $(Au_1,u_1)=0$ in view of skew-symmetricity of the operator $A$ and that $B\subset A^*$. It follows from (\ref{sk2}) that the operator $\tilde B$ is dissipative. But the operator $B$ is a maximal dissipative operator, and we conclude that $\tilde B=B$. In particular $D(A)+D(B)=D(B)$, that is,
$D(A)\subset D(B)$. Hence, $-A\subset B$.

Conversely, assume that $-A\subset B$. Let $u\in D(A)$, $v\in D(B)$. Then $u+sv\in D(B)$ for all $s\in\R$ and by dissipativity of $B$ the inequality $f(s)\doteq (B(u+sv),u+sv)\le 0$ holds for all real $s$. The function $f(s)$ is quadratic and can be written in the form
\begin{align}\label{sk3}
f(s)=(Bu,u)+s((Bu,v)+(Bv,u))+s^2(Bv,v)=\nonumber\\ s((Bu,v)+(Bv,u))+s^2(Bv,v),
\end{align}
because by the assumption $-A\subset B$ we have the equality $(Bu,u)=-(Au,u)=0$. Notice also that $(Bv,v)\le 0$. From representation (\ref{sk3}) and the condition $f(s)\le 0$ $\forall s\in\R$ it follows that $(Bu,v)+(Bv,u)=0$. Hence, for all $v\in D(B)$, $u\in D(A)$  \[(Au,v)=-(Bu,v)=(Bv,u)=(u,Bv).\]
This identity implies that $v\in D(A^*)$ and $A^*v=Bv$ for all $v\in D(B)$. This means that $B\subset A^*$. The proof is complete.
\end{proof}
We underline that the implication $-A\subset B \Rightarrow B\subset A^*$ was also established by R.S.~Phillips  in \cite[Lemma~1.1.5]{Phil}.

\section{Main results}

By Theorem~\ref{th1} and Lemma~\ref{lem1} generators of contractive semigroups of g.s. of equation (\ref{e1}) are exactly  $m$-dissipative extensions of the operator $-A$. First, we describe skew-symmetric extensions of the operator $A$. Such extensions reduce to symmetric extensions of the symmetric operator $iA$ ($i^2=-1$). Theory of symmetric extensions is well-developed and based on the Cayley transform. In the framework of skew-symmetric operators the Cayley transform $Q=(E+A)(E-A)^{-1}$ is even more natural (it does not require complexification of the space and the operators). Since
\[
\|u\pm Au\|^2=\|u\|^2\pm 2(Au,u)+\|Au\|^2=\|u\|^2+\|Au\|^2, \quad \forall u\in D(A),
\]
the operator $Q$ is an isometry between the spaces $H_-=\Im (E-A)$ and $H_+=\Im (E+A)$ (notice that the spaces $H_\pm$ are closed due to the closedness of operator $A$). The Cayley transform is invertible, the inverse transform is defined by the formula $A=(Q-E)(Q+E)^{-1}$ (under the assumption that the domain
$D(A)=\Im(Q+E)$ is dense the operator $Q+E$ is always invertible). Hence, skew-symmetric extensions of the operator $A$ corresponds to isometric extensions of the operator $Q$, which are reduced to construction of isometric maps from $(H_-)^\perp$ into $(H_+)^\perp$. Remind that Hilbert dimensions of these spaces (i.e., Hilbert codimensions of $H_\pm$) are called the deficiency indexes of $A$. We denote them $d_+$, $d_-$, respectively, so that $d_\pm=d_\pm(A)=\codim H_\pm$. Maximal skew-symmetric operators $A$ are characterized by the condition that at least one of the deficiency indexes $d_\pm(A)$ is null. The condition $d_+=d_-=0$ describes skew-adjoint operators, their Cayley transforms $Q$ are orthogonal operators. It is known that any skew-symmetric operator can be extended to a maximal skew-symmetric operator. Using this property, we can now prove the following result.

\begin{theorem}\label{th2}
There exists a contractive semigroup $u(t)=e^{tB}u_0$ of g.s. to equation (\ref{e1}). Moreover, the following two cases are possible:

(i) The generator $B$ is skew-symmetric. In this case the operators $T_t=e^{tB}$ are isometric embeddings and the conservation of energy $\|u(t)\|=\|u_0\|$ holds for all $t>0$;

(ii) the adjoint operator $B^*$ is skew-symmetric. Then the operators $(T_t)^*$ are isometric embeddings.
\end{theorem}

\begin{proof}
The skew-symmetric operator $A$ admits a maximal skew-symmetric extension $\tilde A$. Let $d_\pm=\codim (E\pm\tilde A)$ be the deficiency indexes of $\tilde A$. By the maximality, one of these indexes is null. If $d_+=0$, then $\Im(E+\tilde A)=H$ and therefore the operator $-\tilde A$ is $m$-dissipative. In this case we set $B=-\tilde A$. Similarly, if $d_-=0$ then the operator $\tilde A$ is $m$-dissipative. It is known (cf. \cite[Theorem~1.1.2]{Phil}) that the adjoint operator $(\tilde A)^*$ is $m$-dissipative as well, and $-\tilde A\subset(\tilde A)^*$ by skew-symmetricity of $\tilde A$. Setting $B=(\tilde A)^*$,
we obtain a $m$-dissipative extension of $-A$ in the case $d_-=0$. By Theorem~\ref{th1} the corresponding semigroup $e^{tB}$ provides g.s. of equation (\ref{e1}). Properties (i), (ii) follows from the fact that the semigroup $e^{tB}$ consists of isometric embeddings whenever the generator $B$ is skew-symmetric. We also used that $B^*$ is a generator of the adjoint semigroup $(T_t)^*$.
\end{proof}
We underline that in paper \cite{Phil} existence of an $m$-dissipative extensions was established by different methods,
with the help of the Cayley transform for dissipative operators.

Concerning uniqueness of the constructed in Theorem~\ref{th2} semigroup, the following statement holds.

\begin{theorem}\label{th3}
A contractive semigroup of g.s. to equation (\ref{e1}) is unique if and only if the skew-symmetric operator $A$ is maximal.
\end{theorem}

\begin{proof}
Let the skew-symmetric operator $A$ be maximal and $d_\pm$ be its deficiency indexes. If $u(t)=e^{tB}u_0$ is a contracrive semigroup of g.s. then the infinitesimal generator $B$ is $m$-dissipative and, by Theorem~\ref{th1} and Lemma~\ref{lem1},
$-A\subset B\subset A^*$. By the maximality, one of the operators $-A$, $A^*$ is $m$-dissipative, $-A$ in the case
$d_+=0$ and $A^*$ if $d_-=0$. Since $B$ is an $m$-dissipative operator as-well, we conclude that necessarily $B=-A$ if $d_+=0$ and $B=A^*$ if $d_-=0$ (in the case $d_+=d_-=0$ we have $B=-A=A^*$). Thus, the operator $B$ is uniquely defined, which implies uniqueness of the corresponding semigroup.

Conversely, if the operator $A$ is not maximal then both its deficiency indexes are not zero. Evidently, then there exist different maximal skew-symmetric extensions $\tilde A_1$, $\tilde A_2$ of $A$ with the same deficiency indexes $d_\pm$. The corresponding $m$-dissipative operators
\[
B_1=\left\{\begin{array}{lcr} -\tilde A_1 & , & d_+=0, \\ (\tilde A_1)^* & , & d_-=0, \end{array}\right. \quad B_2=\left\{\begin{array}{lcr} -\tilde A_2 & , & d_+=0, \\ (\tilde A_2)^* & , & d_-=0 \end{array}\right.
\]
are different and therefore generate different contractive semigroups of g.s.
\end{proof}

Now we are ready to formulate a criteria of uniqueness of g.s. to problem (\ref{e1}), (\ref{c1}).

\begin{theorem}\label{th4} A g.s. of problem (\ref{e1}), (\ref{c1}) is unique if and only if the deficiency index $d_-$ of the operator $A$ is zero.
\end{theorem}

\begin{proof}
Assume that a g.s. of problem (\ref{e1}), (\ref{c1}) is unique. We have to prove that $d_-=0$. Supposing the contrary,  $d_->0$, we conclude that the closed subspace $\Im (E-A)$ is strictly contained in $H$. Then $\ker(E-A^*)=(\Im(E-A))^\perp\not=\{0\}$ and there exists a nonzero vector $u_0\in D(A^*)$ such that $A^*u_0=u_0$. The function $u=e^t u_0$ is a g.s. of (\ref{e1}), (\ref{c1}) different from the semigroup g.s. $T_tu_0$ (constructed in Theorem~\ref{th2}), because the latter is bounded. Moreover, for each $t_0\ge 0$ the functions
\[
\tilde u(t)=\left\{\begin{array}{lcr} e^tu_0, & , & 0\le t\le t_0, \\ e^{t_0}T_{t-t_0}u_0 & , & t\ge t_0 \end{array} \right.
\]
are pairwise different bounded g.s. of problem (\ref{e1}), (\ref{c1}), and we come to a contradiction even with the condition of uniqueness of a bounded (!) g.s. Hence, $d_-=0$.

Conversely, assume that $d_-=0$, and that $u(t)$ is a g.s. of problem (\ref{e1}), (\ref{c1}) with zero initial data.
By the linearity, it is sufficient to prove that $u(t)\equiv 0$. Since $d_+(-A)=d_-(A)=d_-=0$ then, in correspondence with   the proof of Theorem~\ref{th2}, the operator $A$ generates the semigroup $u=e^{tA}u_0$ of g.s. to the equation ${u'+A^*u=0}$. Let $v_0\in D(A)$ and $v(t)=e^{(t_0-t)A}v_0$, $t\le t_0$. Then $v(t)\in C^1([0,t_0],H)\cap C([0,t_0],D(A))$, and $v'(t)=-Av(t)$.
Further, we choose such a function $\beta(s)\in C_0(\R)$ that $\supp\beta(s)\subset [-1,0]$, $\beta(s)\ge 0$, $\displaystyle\int\beta(s)ds=1$, and set for ${\nu\in\N}$  \
\[
\beta_\nu(s)=\nu\beta(\nu s), \quad \theta_\nu(t)=\int_t^{+\infty} \beta_\nu(s)ds.
\]
Obviously, $\beta_\nu(s)\in C_0(\R)$, $\supp\beta_\nu(s)\subset [-1/\nu,0]$, $\beta_\nu(s)\ge 0$,
$\displaystyle\int\beta_\nu(s)ds=1$. Therefore, the sequence $\beta_\nu(s)$ converges as $\nu\to\infty$ to the Dirac
$\delta$-measure in $\D'(\R)$ while the functions $\theta_\nu(t)\in C^1(\R)$ decrease, $\theta_\nu(t)=1$ for $t\le -1/\nu$, $\theta_\nu(t)=0$ for $t\ge 0$, and the sequence $\theta_\nu(t)$ converges pointwise as $\nu\to\infty$ to the function
$\theta(-t)$, where $\displaystyle\theta(s)=\left\{\begin{array}{ll} 0, & s\le 0, \\ 1, & s>0.
\end{array}\right. $ is the Heaviside function. The function $f(t)=v(t)\theta_\nu(t-t_0)$ lies in the space  $C_0^1(\R_+,H)\cap C_0(\R_+,D(A))$ (we agree that $f(t)=0$ for $t>t_0$), and by Remark~\ref{rem1}(2) it is a proper test function in relation (\ref{gs}) for the g.s. $u(t)$. Revealing this relation and taking into account the equality $f'(t)=v'(t)\theta_\nu(t-t_0)-v(t)\beta_\nu(t-t_0)=-Af(t)-v(t)\beta_\nu(t-t_0)$, we arrive at the equality
\[
-\int (u(t),v(t))\beta_\nu(t-t_0)dt=0.
\]
Since the scalar product $(u(t),v(t))$ is a continuous function (recall that the $H$-valued functions $v(t)$, $u(t)$ are, respectively, strongly and weakly continuous, cf. Remark~\ref{rem1}(1)), this equality in the limit as $\nu\to\infty$ implies that $(u(t_0),v_0)=(u(t_0),v(t_0))=0$ for each $v_0\in D(A)$. By the density of $D(A)$ in $H$, we conclude that $u(t_0)=0$.  Since $t_0>0$ is arbitrary, $u(t)\equiv 0$, as was to be proved.
\end{proof}

The backward Cauchy problem for equation (\ref{e1}), considered for the time $t<T$ with the Cauchy data
$u(T,x)=u_0(x)$ at the final moment $t=T$, after the change $t\to T-t$ is reduced to the standard problem (\ref{e1}), (\ref{c1}) but related to the operator $-A$ instead of $A$. Therefore, the uniqueness of g.s. to the backward Cauchy problem is equivalent to the condition $d_+=d_-(-A)=0$. In particular, the uniqueness of g.s. for both forward and backward
problems is equivalent to the requirement $d_+=d_-=0$, that is, to the condition that the operator $A$ is skew-adjoint.
We have proved the following

\begin{corollary}\label{cor1} Uniqueness of g.s. for both forward and backward
problems (\ref{e1}), (\ref{c1}) is equivalent to skew-adjointness of the operator $A$.
\end{corollary}

\section{Conclusion}
We return to the case of the transport problem (\ref{tre}), (\ref{tri}), corresponding to the skew-symmetric operator $A_0 u(x)=a(x)\cdot\nabla u(x)$. Uniqueness of g.s. to this problem (both forward and backward) is well known in the case of smooth and bounded coefficients $a_i(x)$, $i=1,\ldots,n$, and in view of Corollary~\ref{cor1} the operator $A$ is skew-adjoint. Uniqueness of g.s. to the forward and backward Cauchy problems is established in celebrated paper by R.J.~DiPerna and P.L.~Lions \cite{DiL} for the coefficients, which have Sobolev regularity, later these results were extended to more general case  of $BV$-regular coefficients in the paper by L.~Ambrosio \cite{Amb}. In correspondence with Corollary~\ref{cor1} in each of these cases the operator $A$ is skew-adjoint. However, in the general case there are several examples of non-uniqueness of g.s. to the problem (\ref{tre}), (\ref{tri}), see \cite{Aiz,Br,CLR,PaTr}, so that skew-adjointness of the operator $A$ may fail.

Remark also that in the case of bounded fields of coefficients the statement of Corollary~\ref{cor1} applied to the problem  (\ref{tre}), (\ref{tri}) follows from results of \cite{BoCr}, see also \cite{prep,ufa}.


\begin{thebibliography}{99}

\bibitem{Aiz}
M. Aizenman, On vector fields as generators of flows. A counterexample
to Nelson's conjecture, Ann. of Math. 107 (1978) 287-296.
\bibitem{Amb}
L. Ambrosio, Transport equation and Cauchy problem for BV vector
fields, Invent. Math. 158 (2004) 227--260.
\bibitem{BoCr}
F. Bouchut, G. Crippa, Uniqueness, renormalization and smooth
approximations for linear transport equations, SIAM J. Math. Anal. 38 (2006) 1316--1328.
\bibitem{Br}
A. Bressan, An ill-posed Cauchy problem for a hyperbolic system in two
space dimensions, Rend. Sem. Mat. Univ. Padova 110 (2003) 103--117.
\bibitem{CLR}
F. Colombini, T. Luo, and J. Rauch, Uniqueness and nonuniqueness
for nonsmooth divergence free transport, S\'emin. \'Equ. D\'eriv. Partielles,
\'Ec. Polytech., Cent. Math., Palaiseau 2002-2003, Exp. no. XXII, 21 p.
(2003).
\bibitem{DiL}
R.J.~DiPerna, P.L.~Lions, Ordinary differential equations, transport
theory and Sobolev spaces, Invent. Math. 98 (1989) 511--547.
\bibitem{PaTr}
E.Yu. Panov, Generalized solutions of the Cauchy problem for a transport equation
with discontinuous coefficients, in: C. Bardos, A.V. Fursikov (Eds.), Instability in Models Connected with Fluid Flows. II, International Mathematical Series, Vol. 7, Springer, New York, 2008, pp.~23--84.
\bibitem{prep}
E.Yu. Panov, On one criterion of the uniqueness of generalized solutions for linear
transport equations with discontinuous coefficients, Arxiv:1504.00836, 17 Apr
2015.
\bibitem{ufa}
E.Yu. Panov, On generalized solutions to linear transport equations with
discontinuous coefficients. Mathematical Modeling of Processes and Systems. Collective monograph. S.A. Mustafina ed. Sterlitamak, 2018, pp. 18--51.
\bibitem{Phil}
R.S. Phillips, Dissipative operators and hyperbolic systems of partial differential equations, Trans. Amer. Math. Soc., Vol. 90 (1959) 193--254.
\end{thebibliography}
\end{document}